\documentclass[11pt]{amsart}

\usepackage{pdfsync}
\usepackage{latexsym}
\usepackage{amsgen,amsmath,amsxtra,amssymb,amsfonts,amsthm}
\usepackage{times}
\usepackage{graphicx}
\usepackage[dvips]{epsfig}
\usepackage{subfigure}
\usepackage[active]{srcltx}
\usepackage{color}
\usepackage{bbm}
\usepackage{accents}

\begin{document}
\newcommand{\dyle}{\displaystyle}
\newcommand{\R}{{\mathbb{R}}}
\newcommand{\ubar}[1]{\underaccent{\bar}{#1}}
\newcommand{\Hi}{{\mathbb H}}
\newcommand{\Ss}{{\mathbb S}}
\newcommand{\N}{{\mathbb N}}
\newcommand{\Rn}{{\mathbb{R}^n}}
\newcommand{\ieq}{\begin{equation}}
\newcommand{\eeq}{\end{equation}}
\newcommand{\ieqa}{\begin{eqnarray}}
\newcommand{\eeqa}{\end{eqnarray}}
\newcommand{\ieqas}{\begin{eqnarray*}}
\newcommand{\eeqas}{\end{eqnarray*}}
\newcommand{\Bo}{\put(260,0){\rule{2mm}{2mm}}\\}
\def\L#1{\label{#1}} \def\R#1{{\rm (\ref{#1})}}


\theoremstyle{plain}
\newtheorem{theorem}{Theorem} [section]
\newtheorem{corollary}[theorem]{Corollary}
\newtheorem{lemma}[theorem]{Lemma}
\newtheorem{proposition}[theorem]{Proposition}
\newtheorem{assumption}[theorem]{Assumption}
\def\neweq#1{\begin{equation}\label{#1}}
\def\endeq{\end{equation}}
\def\eq#1{(\ref{#1})}

\theoremstyle{definition}
\newtheorem{definition}[theorem]{Definition}
\newtheorem{remark}[theorem]{Remark}

\numberwithin{figure}{section}
\newcommand{\res}{\mathop{\hbox{\vrule height 7pt width .5pt depth
0pt \vrule height .5pt width 6pt depth 0pt}}\nolimits}
\def\at#1{{\bf #1}: } \def\att#1#2{{\bf #1}, {\bf #2}: }
\def\attt#1#2#3{{\bf #1}, {\bf #2}, {\bf #3}: } \def\atttt#1#2#3#4{{\bf #1}, {\bf #2}, {\bf #3},{\bf #4}: }
\def\aug#1#2{\frac{\displaystyle #1}{\displaystyle #2}} \def\figura#1#2{ \begin{figure}[ht] \vspace{#1} \caption{#2}
\end{figure}} \def\B#1{\bibitem{#1}} \def\q{\int_{\Omega^\sharp}}
\def\z{\int_{B_{\bar{\rho}}}\underline{\nu}\nabla (w+K_{c})\cdot
\nabla h} \def\a{\int_{B_{\bar{\rho}}}}
\def\b{\cdot\aug{x}{\|x\|}}
\def\n{\underline{\nu}} \def\d{\int_{B_{r}}}
\def\e{\int_{B_{\rho_{j}}}} \def\LL{{\mathcal L}}
\def\itr{\mathrm{Int}\,}
\def\D{{\mathcal D}}
 \def\tg{\tilde{g}}
\def\A{{\mathcal A}}
\def\S{{\mathcal S}}
\def\H{{\mathcal H}}
\def\M{{\mathcal M}}
\def\T{{\mathcal T}}
\def\U{{\mathcal U}}
\def\N{{\mathcal N}}
\def\I{{\mathcal I}}
\def\F{{\mathcal F}}
\def\J{{\mathcal J}}
\def\E{{\mathcal E}}
\def\F{{\mathcal F}}
\def\G{{\mathcal G}}
\def\HH{{\mathcal H}}
\def\W{{\mathcal W}}
\def\H{\D^{2*}_{X}}
\def\d{d^X_M }
\def\LL{{\mathcal L}}
\def\H{{\mathcal H}}
\def\HH{{\mathcal H}}
\def\itr{\mathrm{Int}\,}
\def\vah{\mbox{var}_\Hi}
\def\vahh{\mbox{var}_\Hi^1}
\def\vax{\mbox{var}_X^1}
\def\va{\mbox{var}}
\def\SS{{\mathcal S}}
 \def\Y{{\mathcal Y}}
\def\length{{l_\Hi}}
\newcommand{\average}{{\mathchoice {\kern1ex\vcenter{\hrule
height.4pt width 6pt depth0pt} \kern-11pt} {\kern1ex\vcenter{\hrule height.4pt width 4.3pt depth0pt} \kern-7pt} {} {} }}
\def\weak{\rightharpoonup}
\def\detu{{\rm det}(D^2u)}
\def\detut{{\rm det}(D^2u(t))}
\def\detvt{{\rm det}(D^2v(t))}
\def\detv{{\rm det}(D^2v)}
\def\uuu{u_xu_yu_{xy}}
\def\uuut{u_x(t)u_y(t)u_{xy}(t)}
\def\uuus{u_x(s)u_y(s)u_{xy}(s)}
\def\uuutn{u_x(t_n)u_y(t_n)u_{xy}(t_n)}
\def\vvv{v_xv_yv_{xy}}
\newcommand{\ave}{\average\int}

\title[Radial biharmonic $k-$Hessian equations]{Radial biharmonic $k-$Hessian equations: The critical dimension}

\author[C. Escudero, P. J. Torres]{Carlos Escudero, Pedro J. Torres}
\address{}
\email{}

\date{\today}

\begin{abstract}
This work is devoted to the study of radial solutions to the elliptic problem
\begin{equation}\nonumber
\Delta^2 u = (-1)^k S_k[u] + \lambda f, \qquad x \in B_1(0) \subset \mathbb{R}^N,
\end{equation}
provided either with Dirichlet boundary conditions
\begin{eqnarray}\nonumber
u = \partial_n u = 0, \qquad x \in \partial B_1(0),
\end{eqnarray}
or Navier boundary conditions
\begin{equation}\nonumber
u = \Delta u = 0, \qquad x \in \partial B_1(0),
\end{equation}
where the $k-$Hessian $S_k[u]$ is the $k^{\mathrm{th}}$ elementary symmetric polynomial of eigenvalues of the Hessian matrix
and the datum $f \in L^1(B_1(0))$. We also study the existence of entire solutions to this partial differential equation
in the case in which they are assumed to decay to zero at infinity and under analogous conditions of summability on the datum.
Our results illustrate how, for $k=2$, the dimension $N=4$ plays the role of critical dimension separating two different phenomenologies
below and above it.
\end{abstract}
\renewcommand{\thefootnote}{\fnsymbol{footnote}}
\setcounter{footnote}{-1}
\footnote{This work has been partially supported by the Government of Spain (Ministry of Economy, Industry and Competitiveness)
through Project MTM2015-72907-EXP.}
\renewcommand{\thefootnote}{\arabic{footnote}}
\maketitle

\section{Introduction}

The main objective of this paper is to study radial solutions to elliptic equations of the form
\begin{equation}\label{rkhessian}
\Delta^2 u = (-1)^k S_k[u] + \lambda f, \qquad x \in B_1(0) \subset \mathbb{R}^N,
\end{equation}
where $N, \, k \, \in \mathbb{N}$, $\lambda \in \mathbb{R}$ and $f: B_1(0) \subset \mathbb{R}^N \longrightarrow \mathbb{R}$ is absolutely
integrable. The nonlinearity $S_k[u]$ in~\eqref{rkhessian} is the $k-$Hessian, that is, the sum of the $k^{\mathrm{th}}$ principal minors of the Hessian
matrix $(D^2 u)$. For $k=1$ equation~\eqref{rkhessian} becomes linear, while it is semilinear in the range $2 \le k \le N$.
We will also be interested in finding entire solutions in the whole of $\mathbb{R}^N$ under the assumption $f \in L^1(\mathbb{R}^N)$.

The motivation for studying these equations is multiple. These models arise naturally in the fields of analysis of partial differential equations
and condensed matter physics~\cite{n0,escudero,n1,n2,n3,n4,escudero2,n5,n6}. In this Introduction however we will not review these ideas and discuss
instead an intriguing connection to geometry. Critical points of the Willmore functional are critical points of
the functional
$$
\mathcal{W}= \int_\Sigma H^2 \, d\omega,
$$
where $\Sigma$ is a surface embedded in $\mathbb{R}^3$, $H$ is its mean curvature, and $d\omega$ is the area form of
the surface. Finding these critical points is a classical problem in Differential Geometry~\cite{bryant,GGS}. Now
consider the perturbed Willmore problem
$$
\overset{\sim}{\mathcal{W}}= \int_\Sigma \left( H + H^2 \right) d\omega,
$$
take $\Sigma$ to be the graph of a function $u:B_1(0) \longrightarrow \mathbb{R}$, consider the quasi-flat case $|\nabla u| \ll 1$,
and compute the corresponding Euler-Lagrange equation. For suitable sets of boundary conditions one finds this equation is~\eqref{rkhessian}
for $N=k=2$. So in a sense equation~\eqref{rkhessian} is a generalization of this geometric problem. While this general problem may have a
geometric meaning and connections to Physics~\cite{n3}, we will not explore these herein and instead focus on analytical issues.

To be precise, in this work we will study the biharmonic boundary value problem
\begin{eqnarray}\label{dirichlet}
\Delta^2 u = (-1)^k S_k[u] + \lambda f, \qquad x &\in& B_1(0) \subset \mathbb{R}^N, \\ \nonumber
u = \partial_n u = 0,
\qquad x &\in& \partial B_1(0),
\end{eqnarray}
which we refer to as the Dirichlet problem for partial differential equation~\eqref{rkhessian}, as well as
\begin{eqnarray}\label{navier}
\Delta^2 u = (-1)^k S_k[u] + \lambda f, \qquad x &\in& B_1(0) \subset \mathbb{R}^N, \\ \nonumber
u = \Delta u = 0,
\qquad x &\in& \partial B_1(0),
\end{eqnarray}
which we refer to as the Navier problem for partial differential equation~\eqref{rkhessian}. Also, we will be interested in
finding solutions to
\begin{eqnarray}\label{navier}
\Delta^2 u = (-1)^k S_k[u] + \lambda f, \qquad x &\in& \mathbb{R}^N, \\ \nonumber
u \to 0, \qquad |x| &\to& \infty,
\end{eqnarray}
the entire solutions to this partial differential equation. In every case we will consider radial solutions and, except in Section~\ref{conclusions},
we will limit ourselves to the case $k=2$. This work is organized in the following way. In Section~\ref{radialp} we will describe the radial
problems to be studied and our main results will be described. In Section~\ref{lambda0} we state and prove our existence and non-existence theory
for $\lambda=0$ and the problems with $\lambda \neq 0$ are analyzed in Section~\ref{lneq0}. Finally, our main conclusions,
a non-existence result concerning the case $k=3$ and $\lambda=0$, and a series of open questions are included in Section~\ref{conclusions}.

\section{Radial problems}\label{radialp}

The radial problem corresponding to equation~\eqref{rkhessian} reads
\begin{equation}\label{radialhessian}
\frac{1}{r^{N-1}} [r^{N-1} (\Delta_r u)']' = \frac{(-1)^k}{k} \binom{N-1}{k-1} \frac{1}{r^{N-1}} [r^{N-k}(u')^k]' + \lambda f(r),
\end{equation}
where the radial Laplacian $\Delta_r(\cdot)=\frac{1}{r^{N-1}} [r^{N-1} (\cdot)']'$.
Now integrating with respect to $r$, applying the boundary condition $u'(0)=0$
(the other boundary conditions for the Dirichlet problem are $u(1)=u'(1)=0$),
and substituting $v=u'$ we arrive at
\begin{equation}\nonumber
v'' + \frac{N-1}{r} v' - \frac{N-1}{r^2} v = \frac{(-1)^k}{k} \binom{N-1}{k-1} \frac{v^k}{r^{k-1}}
+ \frac{\lambda}{r^{N-1}} \int_0^{r} f(s) s^{N-1} ds,
\end{equation}
subject to the boundary conditions $v(0)=v(1)=0$.
The change of variables $w(t)=-v(e^{-t})$ leads to the boundary value problem
\begin{equation}
\nonumber
\left\{ \begin{array}{rcl}
-w'' + (N-2)w' + (N-1)w &=& \frac{1}{k} \binom{N-1}{k-1} e^{(k-3)t} w^k \\
& & + \lambda e^{(N-3)t} \int_0^{e^{-t}} f(s) s^{N-1} ds, \\
w(0)=w(+\infty) &=& 0,
\end{array}\right.
\end{equation}
where $k, N \in \mathbb{N}$, $2 \le k \le N$, and $t \in [0,+\infty[$.
This problem can be restated as
\begin{equation}
\label{dirichletr}
\left\{ \begin{array}{rcl}
-w'' + (N-2)w' + (N-1)w &=& \frac{1}{k} \binom{N-1}{k-1} e^{(k-3)t} w^k \\
& & + \lambda e^{(N-3)t} \int_0^{e^{-t}} g(s) \, ds, \\
w(0)=w(+\infty) &=& 0,
\end{array}\right.
\end{equation}
where $g(\cdot) \in L^1([0,1])$.

The problem corresponding to Navier boundary conditions in the radial setting reads
\begin{equation}
\label{navierr}
\left\{ \begin{array}{rcl}
-w'' + (N-2)w' + (N-1)w &=& \frac{1}{k} \binom{N-1}{k-1} e^{(k-3)t} w^k \\
& & + \lambda e^{(N-3)t} \int_0^{e^{-t}} g(s) \, ds, \\
w'(0) - (N-1) w(0) = w(+\infty) &=& 0.
\end{array}\right.
\end{equation}

In this work we will also consider equation~\eqref{rkhessian} in the whole of $\mathbb{R}^N$ subject to the ``boundary condition''
$u \to 0$ when $|x| \to \infty$. We will look for radial solutions to it, which we will call \emph{entire solutions}.
After the same changes of variables above we arrive at the problem
\begin{equation}
\label{entire}
\left\{ \begin{array}{rcl}
-w'' + (N-2)w' + (N-1)w &=& \frac{1}{k} \binom{N-1}{k-1} e^{(k-3)t} w^k \\
& & + \lambda e^{(N-3)t} \int_0^{e^{-t}} g(s) \, ds, \\
w(-\infty)=w(+\infty) &=& 0,
\end{array}\right.
\end{equation}
where in this case we need to assume $g(\cdot) \in L^1([0,\infty[)$.

We note that in~\cite{n6} we proved a series of results concerning some of these problems.
While for the precise statements we refer the reader to this article, we summarize now some of them.
We proved existence of at least one solution to problems~\eqref{dirichletr} and~\eqref{navierr} provided $k=2$, $N \in \{2,3\}$ and $|\lambda|$
is small enough. Moreover the solution is unique in a certain neighborhood of the origin.
Non-existence of solution to problems~\eqref{dirichletr} and~\eqref{navierr} was proven for $k=2$, $N \in \{2,3\}$ and $\lambda$ large enough.
Existence of at least one solution to problems~\eqref{dirichletr} and~\eqref{navierr} provided $k=2$ and $N \in \{2,3\}$
was proven for $\lambda < 0$ independently of the size of $|\lambda|$.

In this paper we address a series of related questions that equally concern the analysis of the existence/non-existence of solutions to the
mentioned boundary value problems. We divided our results into two blocks. The first one concerns the autonomous cases, that is,
$\lambda=0$ is set in the equations. This simplification allows as to prove:

\begin{itemize}
\item Non-existence of non-trivial solutions to problems~\eqref{dirichletr} and~\eqref{navierr} for $k=2$ and $N \ge 4$.
\item Existence of an explicit continuum of non-trivial solutions to problem~\eqref{entire} for $k=2$ and $N = 4$.
\item Existence of a continuum of non-trivial entire solutions for $k=2$ and $N \ge 5$.
\item Non-existence of non-trivial solutions to problem~\eqref{entire} for $k=2$ and $N=2,3$.
\end{itemize}

We extend and complement these results for the cases with $\lambda \neq 0$ in the following sense:
\begin{itemize}
\item We prove existence of isolated solutions for small $|\lambda|$ to problems~\eqref{dirichletr}, \eqref{navierr},
and~\eqref{entire} for $k=2$ and $N \ge 2$.
\item Non-existence of solution to problems~\eqref{dirichletr}, \eqref{navierr}, and~\eqref{entire} for $k=2$ and $N \ge 2$
is proven for $\lambda$ large enough.
\item On the other hand, existence of at least one solution to these problems in this same range of parameters is shown for $\lambda <0$
despite the possible large values of $|\lambda|$.
\end{itemize}

As a consequence, all these results illustrate in what sense $N=4$ is the critical dimension for the radial biharmonic $2-$Hessian problem.

\section{Autonomous formulation and main results for $\lambda=0$}\label{lambda0}

In this section we focus on existence and non-existence results in the case $k=2$ and $\lambda=0$.
We start formulating the problems under consideration in a favorable setting.

By introducing the change of variables
\begin{equation}
\label{change}
w=e^{\gamma t}z
\end{equation}
with $\gamma=\frac{3-k}{k-1}$, the equation
$$
-w'' + (N-2)w' + (N-1)w =\frac{1}{k} \binom{N-1}{k-1} e^{(k-3)t}w^k
$$
is transformed into
\begin{equation}\label{eq-z}
-z''+(N-2-2\gamma)z'+[N-1+\gamma(N-2)-\gamma^2]z=\alpha_{k,N} z^k,
\end{equation}
where $\alpha_{k,N}=\frac{1}{k} \binom{N-1}{k-1}$. By the
autonomous character of equation~\eqref{eq-z}, the different boundary
value problems can be studied by a suitable analysis of the phase
plane.

\subsection{Dirichlet problem}

When $k=2$ the Dirichlet problem associated to eq. \eqref{eq-z} reads
\begin{equation}
\label{dirichletr-2-0} \left\{ \begin{array}{rcl}
-z'' + (N-4)z' + (2N-4)z &=& \alpha_{2,N} z^2 \\
z(0)=z(+\infty) &=& 0.
\end{array}\right.
\end{equation}

A trivial consequence of the change introduced by \eqref{change} is as follows.

\begin{lemma}\label{lem1}
If $w$ is a non-trivial solution of \eqref{dirichletr} with $k=2$,
then $z=e^{- t}w$ is a non-trivial solution of
\eqref{dirichletr-2-0}.
\end{lemma}

\begin{remark}
Note that the reciprocal statement to that of Lemma~\ref{lem1} is not true, i.~e. a nontrivial solution to equation~\eqref{dirichletr-2-0}
does not necessarily give rise to a solution of~\eqref{dirichletr}. In particular, the boundary conditions could not be obeyed if
$z(t)$ does not decay fast enough when $t \to \infty$.
\end{remark}

\begin{theorem}\label{th-k=2-Dir}
Let us assume that $k=2$, $\lambda=0$ and $N\geq 4$. Then, the
unique solution of problem \eqref{dirichletr} is the trivial one.
\end{theorem}

\begin{proof}
In view of Lemma \ref{lem1}, we only have to prove that the unique
solution of problem \eqref{dirichletr-2-0} is the trivial one.

If $N=4$, the equation is
$$
-z'' + 4z = \frac{3}{2}z^2,
$$
that has the conserved quantity
$$
V(z,z')=\frac{z'^2}{2}-2z^2+\frac{1}{2}z^3.
$$
Orbits such that $z(+\infty)=0$ belong to the energy level $V(z,z')=0$. Then, if $z(0)=0$ one gets $z'(0)=0$ and the solution is the trivial one.

In the case $N>4$, there is not a conserved quantity anymore, but it is still possible to perform a basic analysis of the phase plane to discard the existence of non-trivial solutions. Taking $y=z'$, the equation is equivalent to the planar system
\begin{equation}
\label{planar-syst}
\left\{\begin{array}{rcl}
z' & = & y\\
y' & = & (N-4)y + (2N-4)z - \alpha_{2,N} z^2.
\end{array}\right.
\end{equation}
This system has two equilibria, the origin and the point
$$
\gamma_1=\left(\frac{2N-4}{\alpha_{2,N}},0\right). $$ The origin is a saddle
point, hence an orbit such that $z(+\infty)=0$ must belong to the
stable manifold of codimension one. A basic local analysis shows
that the stable manifold approach the origin on the direction
$(1,-2)$, which is the eigenvector associated to the negative
eigenvalue of the linearized system. An eventual non-trivial
solution of problem \eqref{dirichletr} would mean that the stable
manifold cuts the vertical axis. Then, such piece of stable
manifold and the corresponding segment of vertical axis would
define a positively invariant compact set $C$ (see Fig. \ref{graf1}). Therefore, any
orbit starting on $C$ has non-empty $\omega$-limit set, but due to
the dissipation, system \eqref{planar-syst} does not have closed
orbits (or closed paths composed by orbits). Then, the only
candidate to be the $\omega$-limit set is the point $\gamma_1$,
but it is easy to verify that $\gamma_1$ is an unstable node. The
contradiction comes from assuming the existence of a non-trivial
solution of problem \eqref{dirichletr}.
\end{proof}

%
\begin{figure}
\includegraphics[scale=0.5]{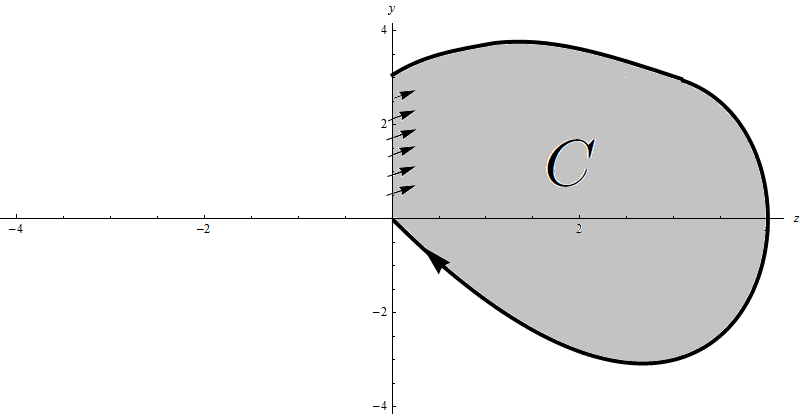}
%
%
\caption{The compact set $C$.}
\label{graf1}       
\end{figure}

\subsection{Navier problem}

When $k=2$ the Navier problem reads
\begin{equation}
\label{navierr-2-0} \left\{ \begin{array}{rcl}
-z'' + (N-4)z' + (2N-4)z &=& \alpha_{2,N} z^2 \\
z'(0)-(N-2)z(0)=z(+\infty) &=& 0.
\end{array}\right.
\end{equation}
In this case we find analogous results to those proven in the previous section.

\begin{theorem}\label{th-k=2-Na}
Let us assume that $k=2$, $\lambda=0$ and $N\geq 4$. Then, the
unique solution of problem \eqref{navierr} is the trivial one.
\end{theorem}

\begin{proof}
As in Theorem \ref{th-k=2-Dir}, we only have to prove that the
unique solution of problem \eqref{navierr-2-0} is the trivial one.

For $N=4$, we can use again an argument based on the conservation
of energy. Looking at the planar system \eqref{planar-syst}, the
quantity
$$
V(z,y)=\frac{y^2}{2}-2z^2+\frac{1}{2}z^3
$$
is constant along orbits. An eventual solution of problem
\eqref{navierr} would mean that the line $y-2z=0$ intersects the
curve $V(z,y)=0$ at a point different to the origin. It is simple
to see that this is not the case.

In the case $N>4$, the proof is analogous to that of Theorem
\ref{th-k=2-Dir}.
\end{proof}

\begin{remark}
A simple re-scaling of the spatial variable shows that Theorems
\ref{th-k=2-Dir} and \ref{th-k=2-Na} are valid not only for the
Dirichlet and Navier problems posed on the unit ball, but on
\emph{any} ball.
\end{remark}

\subsection{Entire solutions}
\label{entiresol}

In this subsection, we consider problem \eqref{entire} with $k=2$ and $\lambda=0$ , that is
\begin{equation}
\label{entire-0}
\begin{array}{rcl}
-w'' + (N-2)w' + (N-1)w &=& \frac{N-1}{2}  e^{-t} w^2 \\
w(-\infty)=w(+\infty) &=& 0.
\end{array}
\end{equation}
Introducing the change
$$
w=e^t z,
$$
the equation becomes autonomous and reads
\begin{equation}\label{entire2}
-z'' + (N-4)z' + (2N-4)z =\frac{N-1}{2}  z^2,
\end{equation}
with boundary conditions
\begin{equation}\label{entire2-bc}
\lim_{t\to-\infty} e^t z(t)=0=\lim_{t\to+\infty} e^t z(t).
\end{equation}

A first observation is that our previous results for the Dirichlet and Navier problems, valid for any ball, cannot be  extended to the whole of $\mathbb{R}^N$.
In particular, our first task is to find an explicit entire solution when $N=4$.

For $N=4$, equation \eqref{entire2} reads
\begin{equation}\label{entire2-4}
-z'' + 4z = \frac{3}{2}z^2,
\end{equation}
that has the following family of exact solutions
$$
z(t)= \frac{16 \, e^{2(t-t_0)}}{\left[1+e^{2(t-t_0)}\right]^2},
$$
where $t_0 \in \mathbb{R}$ is arbitrary. This family of solutions corresponds to the unique homoclinic orbit of the phase plane (see Fig. \ref{graf2}), of course they hold the conditions  $z(-\infty)=z(+\infty)=0$, but it is easy to verify that conditions \eqref{entire2-bc} hold as well. In terms of our original variables we find the explicit solution
\begin{equation}\label{esol}
u(r)=\frac{8}{1+\alpha r^2},
\end{equation}
for any $\alpha \in \mathbb{R}^+$. So we have found a continuum of classical solutions to our partial differential equation. For $\alpha<0$, the functions given by formula~\eqref{esol} present two singularities at $r=\pm\sqrt{-\alpha}$, but such a family can still be seen as
a collection of weak solutions, more concretely belonging to the Lorentz space $L^{1,\infty}(\mathbb{R}_+)$.

%
\begin{figure}[t]
\includegraphics[scale=0.4]{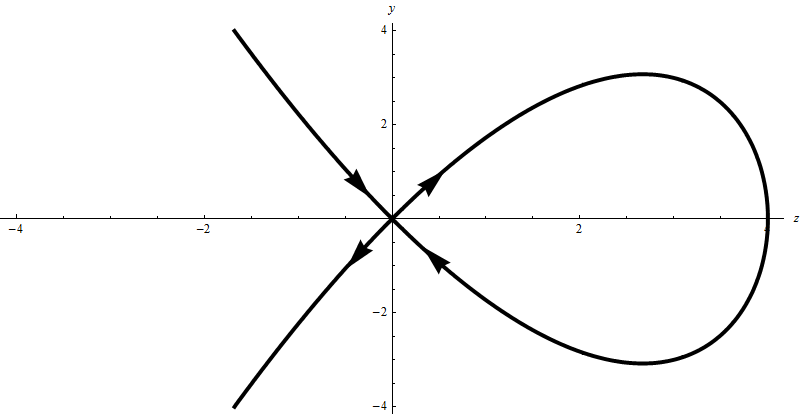}
%
%
\caption{The homoclinic orbit of eq. \eqref{entire2-4}.}
\label{graf2}       
\end{figure}

In the next result, we analyze the case $N>4$. In this case, we cannot give an explicit solution in closed form, but entire solutions still exist.

\begin{theorem}
For any $N>4,$ problem~\eqref{entire2}-\eqref{entire2-bc} has a continuum of non-trivial positive solutions.
\end{theorem}

\begin{proof}
The equivalent planar system to equation \eqref{entire2} (see system \eqref{planar-syst}) was partially studied in the proof of Theorem \ref{th-k=2-Dir}. We know that the equilibrium $(0,0)$ is a saddle point and the second equilibrium $\gamma_1=\left(\frac{2N-4}{\alpha_{2,N}},0\right)$ is an unstable node. In fact, basic arguments show that the stable manifold of $(0,0)$ is a heteroclinic connection with $\gamma_1$ and belong to the semiplane $z>0$ (see Fig. \ref{graf3}). This means that there is a solution $z(t)$ such that
$$
z(-\infty)=\frac{2N-4}{\alpha_{2,N}},\quad z(+\infty)=0.
$$
More precisely, the rate of convergence to $0$ is as $e^{-2t}$ because $\lambda=-2$ is the negative eigenvalue associated to the Jacobian matrix of system \eqref{planar-syst} in $(0,0)$. Therefore, the boundary conditions \eqref{entire2-bc} hold a $z(t)$ is a non-trivial positive solution of the problem. Now, $z(t+t_0)$ for $t_0\in\mathbb{R}$ generates a whole continuum of non-trivial positive solutions.
\end{proof}

%
\begin{figure}[t]
\includegraphics[scale=0.3]{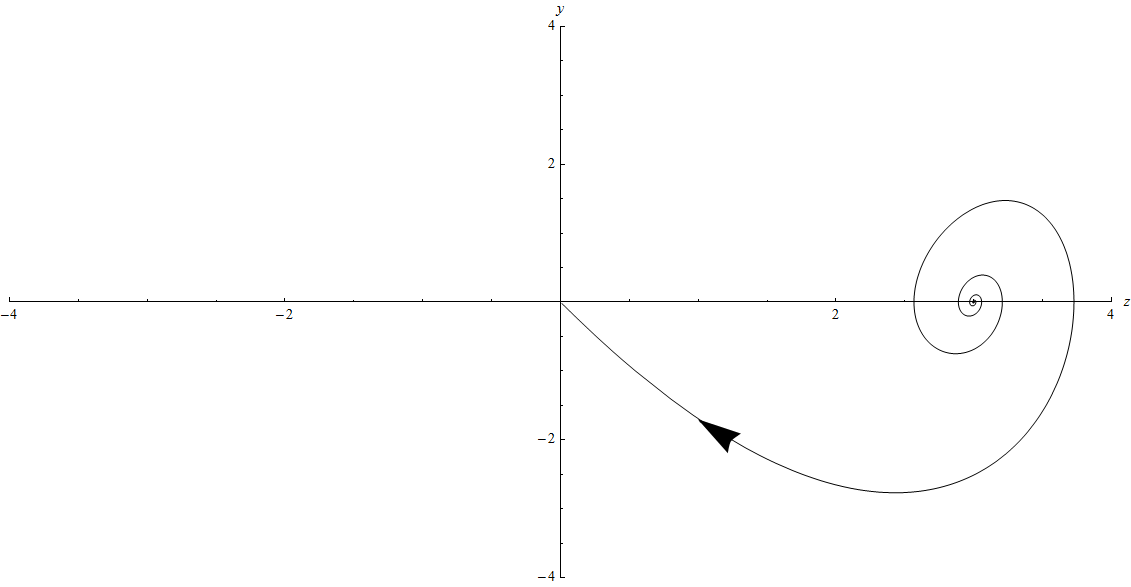}
%
%
\caption{The stable manifold of the equilibrium $(0,0)$.}
\label{graf3}       
\end{figure}

Finally, we prove a non-existence result for the cases $N=2,3$ .

\begin{theorem}
If $N=2,3$, the unique solution of problem~\eqref{entire2}-\eqref{entire2-bc} is the trivial one.
\end{theorem}

\begin{proof}
If $N=3$, there is still an orbit connecting $(0,0)$ and $\gamma_1$, but now it tends to $\gamma_1$ as $t\to+\infty$, hence it does not verify the boundary conditions  \eqref{entire2-bc}. If fact, any solution belonging to the stable manifold of $(0,0)$ is unbounded (the second quadrant is negatively invariant, see \ref{graf4}). Besides, by introducing the change of variables $z(t)=-x(-t)$ into equation  \eqref{entire2}, we get an equivalent equation satisfying the hypotheses of \cite[Theorem 1.1]{LH}. Putting these facts together, any solution $z(t)$ such that $z(+\infty)=0$ is not defined on the whole real line and of course can not verify  \eqref{entire2-bc}. The case $N=2$ is handled by similar arguments.

%
\begin{figure}[t]
\includegraphics[scale=0.4]{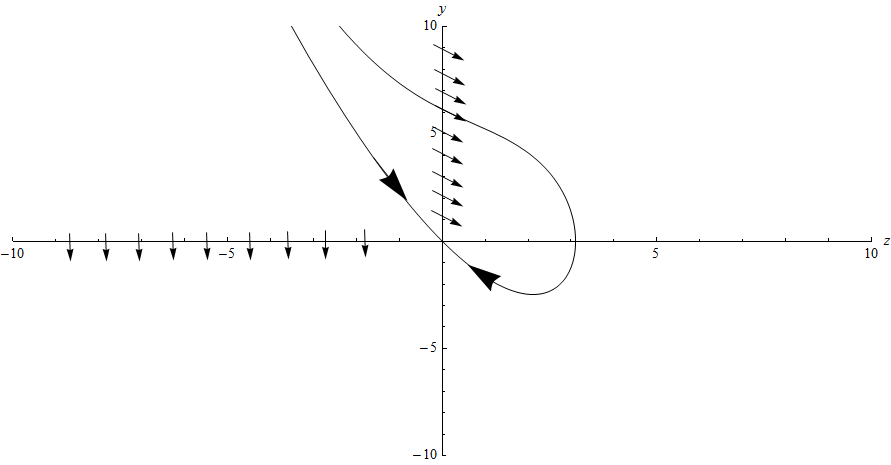}
%
%
\caption{The stable manifold of the equilibrium $(0,0)$ for $N=3$.}
\label{graf4}       
\end{figure}

\end{proof}

\section{The case $\lambda \neq 0$}\label{lneq0}

This section is devoted to study the problems under consideration with $\lambda \neq 0$ but under different assumptions on the value of this parameter.
Moving to the same framework as in the previous section we find
\begin{eqnarray}\nonumber
& & -z''+(N-2-2\gamma)z'+[N-1+\gamma(N-2)-\gamma^2]z \\ \nonumber
&=& \alpha_{k,N} z^k + \lambda e^{(N-3-\gamma)t} \int_0^{e^{-t}} g(s) \, ds.
\end{eqnarray}
When $k=2$ this equation simplifies to
$$
-z'' + (N-4)z' + (2N-4)z = \alpha_{2,N} z^2 + \lambda e^{(N-4)t} \int_0^{e^{-t}} g(s) \, ds.
$$

\subsection{Dirichlet and Navier cases}

In these cases we will prove the existence of a branch of solutions departing from $(z,\lambda)=(0,0)$ under the following assumption on the datum:

\begin{assumption}\label{hypo1}
We assume $g \in L^1 \left([0,1]\right)$ is such that
$$
\lim_{t \to \infty} e^{(N-3)t} \int_0^{e^{-t}} g(s) \, ds =0.
$$
\end{assumption}

\begin{definition}\label{pecbc1}
We define $\tilde{C}^m_0(\mathbb{R}_+)$, $m=0,1,2,\cdots$, to be the set of $m-$times continuously differentiable functions over $\mathbb{R}_+$ such that
$e^t h(t) \to 0$ when $t \to \infty$ for all $h \in \tilde{C}^m_0(\mathbb{R}_+)$.
\end{definition}

\begin{definition}\label{pecbc2}
We define the mapping
\begin{eqnarray}\nonumber
\mathcal{F}: \tilde{C}^2_0(\mathbb{R}_+) \times \mathbb{R} &\longrightarrow& \tilde{C}_0(\mathbb{R}_+) \\ \nonumber
(z,\lambda) &\longmapsto& -z'' + (N-4)z' + (2N-4)z \\ \nonumber
& & - \alpha_{2,N} z^2 - \lambda e^{(N-4)t} \int_0^{e^{-t}} g(s) \, ds.
\end{eqnarray}
\end{definition}

\begin{remark}
The necessity of introducing Definitions~\ref{pecbc1} and~\ref{pecbc2} comes from the peculiar boundary condition $e^t z(t) \to 0$
when $t \to \infty$.
\end{remark}

\begin{theorem}\label{imfudina}
Under hypothesis~\ref{hypo1} and for any $N \ge 2$, there exists a neighborhood $U \subset \tilde{C}^2_0(\mathbb{R}_+)$ of $0$
and an open neighborhood $I \subset \mathbb{R}$ of $0$
such that for every $\lambda \in I$ there is
a unique element $z(\lambda) \in U$ so that $\mathcal{F}[z(\lambda),\lambda]=0$. Moreover the mapping
$$
I \ni \lambda \longmapsto z(\lambda)
$$
is of class $C^1$, $z(0)=0$ and the linear map $\mathcal{F}_z[z(\lambda),\lambda]$ is bijective for any $\lambda \in I$.
\end{theorem}

\begin{proof}
First of all note that for $g$ fulfilling Assumption~\ref{hypo1} functional $\mathcal{F}$ is well defined, and it is also of class $C^1$.
The proof follows from the observation $\mathcal{F}[0,0]=0$,
the fact that the linear map $\mathcal{F}_z[0,0]=-\frac{d^2}{dt^2}+(N-4)\frac{d}{dt}+(2N-4)$
is bijective for both sets of boundary conditions $z(0)=z(+\infty)=0$ and $z'(0)-(N-2)z(0)=z(+\infty)=0$ and any $N \ge 2$
due to the absence of the null eigenvalue in its spectrum and the application of the \emph{implicit function theorem}.
\end{proof}

\begin{remark}
The above result trivially implies the existence of an isolated classical solution for both Dirichlet and Navier problems when $|\lambda|$ is small enough.
\end{remark}

\subsection{Entire solutions}

In this case we can prove the existence of a branch of solutions departing from $(z,\lambda)=(0,0)$ under the following hypothesis:

\begin{assumption}\label{hypo2}
We assume $g \in L^1 \left([0,\infty[\right)$ is such that
$$
\lim_{t \to \pm \infty} e^{(N-3)t} \int_0^{e^{-t}} g(s) \, ds =0.
$$
\end{assumption}

\begin{definition}\label{pecbc3}
We define $\hat{C}^m_0(\mathbb{R})$, $m=0,1,2,\cdots$, to be the set of $m-$times continuously differentiable functions over $\mathbb{R}$ such that
$e^t h(t) \to 0$ when $t \to \pm \infty$ for all $h \in \hat{C}^m_0(\mathbb{R})$.
\end{definition}

\begin{definition}\label{pecbc4}
We define the mapping
\begin{eqnarray}\nonumber
\mathcal{F}_1: \hat{C}^2_0(\mathbb{R}) \times \mathbb{R} &\longrightarrow& \hat{C}_0(\mathbb{R}) \\ \nonumber
(z,\lambda) &\longmapsto& -z'' + (N-4)z' + (2N-4)z \\ \nonumber
& & - \alpha_{2,N} z^2 - \lambda e^{(N-4)t} \int_0^{e^{-t}} g(s) \, ds.
\end{eqnarray}
\end{definition}

\begin{remark}
The necessity of introducing Definitions~\ref{pecbc3} and~\ref{pecbc4} comes from the peculiar boundary condition set $e^t z(t) \to 0$
when $t \to \pm \infty$.
\end{remark}

\begin{theorem}\label{imfues}
Under hypothesis~\ref{hypo2} and for any $N \ge 2$, there exists a neighborhood $U \subset \hat{C}^2_0(\mathbb{R})$ of $0$ and an open neighborhood $I \subset \mathbb{R}$ of $0$
such that for every $\lambda \in I$ there is
a unique element $z(\lambda) \in U$ so that $\mathcal{F}[z(\lambda),\lambda]=0$. Moreover the mapping
$$
I \ni \lambda \longmapsto z(\lambda)
$$
is of class $C^1$, $z(0)=0$ and the linear map $(\mathcal{F}_1)_z[z(\lambda),\lambda]$ is bijective for any $\lambda \in I$.
\end{theorem}

\begin{proof}
Note that for $g$ fulfilling Assumption~\ref{hypo2} functional $\mathcal{F}_1$ is well defined, and it is also of class $C^1$.
The proof follows from the observation $\mathcal{F}_1[0,0]=0$,
the fact that the linear map $(\mathcal{F}_1)_z[0,0]=-\frac{d^2}{dt^2}+(N-4)\frac{d}{dt}+(2N-4)$
is bijective when subject to the boundary conditions $z(\pm \infty)=0$ for any $N \ge 2$
due to the absence of the null eigenvalue in its spectrum and the application of the \emph{implicit function theorem}.
\end{proof}

\begin{remark}
As in the previous cases, this result trivially implies the existence of an isolated classical solution when $|\lambda|$ is small enough.
\end{remark}

\subsection{Non-existence of solution}

Our first result relates to the sharpness of the results in the previous subsections.

\begin{theorem}\label{sharp}
Assumptions~\ref{hypo1} and~\ref{hypo2} are sharp.
\end{theorem}

\begin{proof}
In the case of the entire solutions we re-write problem~\eqref{entire} as
\begin{eqnarray}\nonumber
z(t) &=& \frac{e^{-2t}}{N}\int_{-\infty}^t e^{2s} \, f_1(s) \, ds + \frac{e^{(N-2)t}}{N} \int_t^\infty e^{-(N-2)s} \, f_1(s) \, ds \\ \nonumber
&\ge& \frac{e^{-2t}}{N}\int_{-\infty}^t e^{2s} \, f_2(s) \, ds + \frac{e^{(N-2)t}}{N} \int_t^\infty e^{-(N-2)s} \, f_2(s) \, ds =: z_1(t).
\end{eqnarray}
where
\begin{eqnarray}\nonumber
f_1(t) &=& \alpha_{2,N} [z(t)]^2 +  f_2(t), \\ \nonumber
f_2(t) &=& \lambda e^{(N-4)t} \int_0^{e^{-t}} g(s) \, ds.
\end{eqnarray}
Now, choose $g(s)$ such that
$$
f_2(t)\geq e^{-t} \mathbbm{1}_{[0,\infty[}(t),
$$
or equivalently
$$
 \lambda e^{(N-3)t} \int_0^{e^{-t}} g(s) \, ds\geq 1\qquad \mbox{ for all }t>0.
$$
It is an easy exercise to check that the boundary condition at $+\infty$ is never fulfilled in these cases for $z_1(t)$, in fact $z_1(+\infty)=+\infty$
what in turn forces the same for $z(t)$. Identical results follow for both Dirichlet and Navier problems.
\end{proof}

The following result indicates that, even when we can continue the solution for small $|\lambda|$, this extension cannot be carried out
for an arbitrarily large value of this parameter.

\begin{theorem}
Problems~\eqref{dirichletr}, \eqref{navierr} and~\eqref{entire} have no solutions whenever $g \ge 0$,
$\text{ess sup} \, g >0$, it obeys the corresponding Assumption~\ref{hypo1} or~\ref{hypo2} and $\lambda >0$ is large enough.
\end{theorem}

\begin{proof}
In the case of entire solutions we consider the auxiliary problem
$$
-z'' + (N-4)z' + (2N-4)z = \alpha_{2,N} z^2 + \lambda \, \underbrace{e^{(N-4)t} \int_0^{e^{-t}} g(s) \, ds}_{=: \tilde{f}(t)},
$$
subject to the conditions $z(t) \to 0$ when $t \to -\infty$ and $z(0)=z_0$. This problem can be cast into the form
\begin{eqnarray}\nonumber
z(t) &=& \frac{e^{-2t}}{N} \int_{-\infty}^t e^{2 s} f_1(s) \, ds + e^{(N-2)t} \times \\ \nonumber
& & \left[ z_0 - \frac1N \int_{-\infty}^0 e^{2s} f_1(s) \, ds -\frac1N \int_0^t e^{-(N-2)s} f_1(s) \, ds \right] \\ \nonumber
&\ge& \lambda \frac{e^{-2t}}{N} \int_{-\infty}^t e^{2 s} \tilde{f}(s) \, ds + e^{(N-2)t} \times \\ \nonumber
& & \left[ z_0 - \frac{\lambda}{N} \int_{-\infty}^0 e^{2s} \tilde{f}(s) \, ds -\frac{\lambda}{N} \int_0^t e^{-(N-2)s} \tilde{f}(s) \, ds \right].
\end{eqnarray}
Therefore if
\begin{equation}\label{ineqz01}
z_0 > \lambda \left[ \frac{1}{N} \int_{-\infty}^0 e^{2s} \tilde{f}(s) \, ds + \frac{1}{N} \int_0^\infty e^{-(N-2)s} \tilde{f}(s) \, ds \right],
\end{equation}
then clearly $z(t) \not\to 0$ when $t \to \infty$. Consider now again the original equation
\begin{eqnarray}\nonumber
z(t) &=& \frac{e^{-2t}}{N}\int_{-\infty}^t e^{2s} \, f_1(s) \, ds + \frac{e^{(N-2)t}}{N} \int_t^\infty e^{-(N-2)s} \, f_1(s) \, ds \\ \nonumber
&\ge& \alpha_{2,N} \frac{e^{-2t}}{N}\int_{-\infty}^t e^{2s} \, [z(s)]^2 \, ds \\ \nonumber
& & + \alpha_{2,N} \frac{e^{(N-2)t}}{N} \int_t^\infty e^{-(N-2)s} \, [z(s)]^2 \, ds \\ \nonumber
&\ge& \alpha_{2,N} \frac{e^{-2t}}{N}\int_{-\infty}^t e^{2s} \, [z_1(s)]^2 \, ds \\ \nonumber
& & + \alpha_{2,N} \frac{e^{(N-2)t}}{N} \int_t^\infty e^{-(N-2)s} \, [z_1(s)]^2 \, ds \\ \nonumber
&=& \lambda^2 \left[ \alpha_{2,N} \frac{e^{-2t}}{N}\int_{-\infty}^t e^{2s} \, [z_\lambda(s)]^2 \, ds \right. \\ \nonumber
& & + \left. \alpha_{2,N} \frac{e^{(N-2)t}}{N} \int_t^\infty e^{-(N-2)s} \, [z_\lambda(s)]^2 \, ds \right],
\end{eqnarray}
where $z_\lambda = z_1 / \lambda$. In consequence
\begin{eqnarray}\label{ineqz02}
z(0) &\ge& \lambda^2 \left[ \frac{\alpha_{2,N}}{N}\int_{-\infty}^0 e^{2s} \, [z_\lambda(s)]^2 \, ds \right. \\ \nonumber
& & + \left. \frac{\alpha_{2,N}}{N} \int_0^\infty e^{-(N-2)s} \, [z_\lambda(s)]^2 \, ds \right].
\end{eqnarray}
From~\eqref{ineqz01} and~\eqref{ineqz02} it follows that no solution exists for $\lambda$ large enough.
The proof for both the Dirichlet and Navier problems follows analogously once one chooses the right shooting parameters,
and these parameters can be chosen as in~\cite{n6}.
\end{proof}

\subsection{Negative $\lambda$}

Now we will examine the case when $\lambda <0$ and prove that at least one solution exists always in this range.

\begin{theorem}
Problems~\eqref{dirichletr}, \eqref{navierr} and~\eqref{entire} have at least one solution when $g \ge 0$,
$\text{ess sup} \, g >0$, it obeys the corresponding Assumption~\ref{hypo1} or~\ref{hypo2} and $\lambda <0$.
\end{theorem}

\begin{proof}
The proof is based on the classical method of upper and lower solutions (see for instance \cite{opial}). First, let us note that the differential operator
$$
L[z]:=-z'' + (N-4)z' + (2N-4)z
$$
with the respective boundary conditions of problems ~\eqref{dirichletr}, \eqref{navierr} and~\eqref{entire}, is positively invertible, that is, for any function $f(t)$ positive a.e., the linear equation $Lz=f$  has a unique positive solution $z=L^{-1} f$ with the corresponding boundary conditions. Then, the function $\xi_0(t)=L^{-1}\lambda \tilde{f}$ verifies
$$
L\xi_0=\lambda \tilde{f}\leq \alpha_{2,N} \xi_0^2 +\lambda \tilde{f}
$$
and it is a lower solution of the problem. Moreover, $\alpha_0(t)$ is negative because of the signs assumed for $\lambda$ and $g$. Also, it is direct to check that $\beta\equiv 0$ is an upper solution. Then, the recurrence $\xi_{n+1}=L^{-1}[\alpha_{2,N} \xi_n^2 +\lambda \tilde{f}]$ generates a monotone sequence of lower solutions converging uniformly to a solution of the problem.
\end{proof}

\section{Conclusions, further results, and open questions}\label{conclusions}

This work has been devoted to the analysis of the existence of radial solutions to biharmonic $2-$Hessian problems.
Our results have been divided into two blocks. The first concerns the autonomous problems, that is,
the particular choice $\lambda=0$. In this case we have proven
non-existence of non-trivial solutions to problems~\eqref{dirichletr} and~\eqref{navierr},
i.~e. respectively the Dirichlet and Navier homogeneous boundary value problems, for $k=2$ and $N \ge 4$.
We have also found the existence of a continuum of non-trivial entire solutions,
that is solutions to problem~\eqref{entire}, for $k=2$ and $N \ge 4$;
moreover we have given explicit formulas for them in $N=4$.
Finally we have proven non-existence of non-trivial entire solutions for $k=2$ and $N=2,3$.
These results, together with our previous ones~\cite{n6}, highlight in what sense $N=4$ plays the role of critical dimension for the radial
biharmonic $2-$Hessian equation.

We have complemented this first block with a second block in which the non-autonomous equation is analyzed,
that is, $\lambda \neq 0$ is considered. In this case the existence of isolated solutions for small $|\lambda|$
to problems~\eqref{dirichletr}, \eqref{navierr}, and~\eqref{entire} for $k=2$ and $N \ge 2$ has been proven.
The non-existence of solution to problems~\eqref{dirichletr}, \eqref{navierr}, and~\eqref{entire} for $k=2$ and $N \ge 2$
was proven for $\lambda$ large enough, and this result was supplemented with the corresponding existence one for $\lambda <0$
irrespectively of the size of $|\lambda|$.

There is a series of open questions that arises naturally from the present results.
One is the possible continuation of the non-trivial solutions explicitly computed in Section~\ref{entiresol} when $N=4$ for $\lambda \neq 0$.
Note that in this case we cannot employ the
same arguments as in Theorems~\ref{imfudina} and~\ref{imfues}
since the linearization of the nonlinear operator about these solutions gives rise to the linear operator
$$
-\frac{d^2}{dt^2} + 1 - \frac{48 \, e^{2(t+t_0)}}{(e^{2t} + e^{2t_0})^2}.
$$
The kernel of this operator, provided with the correct boundary conditions, is spanned by the vector
$$
\frac{e^t \left[e^{4t} + e^{4t_0} -3e^{2(t+t_0)}\right]}{(e^{2t} + e^{2 t0})^3},
$$
what rules out the use of the implicit function theorem in this concrete context. Note that this fact is in agreement with the existence of a continuum of
solutions and with the dilatation invariance structure of equation~\eqref{rkhessian}. Different techniques, such as the Lyapunov-Schmidt reduction
or global bifurcation arguments, must be employed in this case.

There are plenty of open questions that concern the cases with $k \ge 3$. Actually, for $k=3$ we have the following result:

\begin{theorem}
Problems~\eqref{dirichletr}, \eqref{navierr}, and~\eqref{entire} have no non-trivial solutions for $k=3$ and $\lambda=0$.
\end{theorem}

\begin{proof} For $k=3$, the relevant equation
$$
-w'' + (N-2)w' + (N-1)w =\frac{1}{k} \binom{N-1}{2} w^3
$$
is autonomous, and in fact an autonomous Duffing oscillator with negative damping~\cite{KB}.
Then an elementary analysis of the phase plane (similar to that of system \eqref{planar-syst}) leads to the conclusion.
\end{proof}

All throughout this work we have illustrated how $N=4$ plays the role of critical dimension for the $2-$Hessian nonlinearity.
This last result on the other hand suggests that the existence of an analogous concept of critical dimension may not be possible
for $k = 3$. In fact, this could be so for any $k > 3$ too, something that would agree with the complementary
results in~\cite{n0} (in particular, these results suggest that there is no critical dimension for $k>3$ and
it takes the marginal value $N=3$ for $k=3$).
We leave this problem as an open question, as well as the extension
of our present results, when possible, to this range of values of $k$.

\vskip20mm
\noindent
{\footnotesize
Carlos Escudero\par\noindent
Departamento de Matem\'aticas\par\noindent
Universidad Aut\'onoma de Madrid\par\noindent
{\tt carlos.escudero@uam.es}\par\vskip1mm\noindent
\& \par\vskip1mm\noindent
Pedro J. Torres\par\noindent
Departamento de Matem\'atica Aplicada\par\noindent
Universidad de Granada\par\noindent
{\tt ptorres@ugr.es}\par\vskip1mm\noindent
}

\begin{thebibliography}{99}

\bibitem{n0} P. Balodis and C. Escudero, \textit{Polyharmonic $k-$Hessian equations in $\mathbb{R}^N$}, arXiv:1603.09392.

\bibitem{bryant} R. L. Bryant, \textit{A duality theorem for Willmore surfaces}, J. Differential Geom. {\bf 20} (1984) 23--53.

\bibitem{escudero} C. Escudero, {\it Geometric principles of surface growth}, Phys. Rev. Lett. {\bf 101} (2008) 196102.

\bibitem{n1} C. Escudero, F. Gazzola, R. Hakl, I. Peral and P.~J. Torres, {\it Existence results for a fourth order partial differential equation arising in
condensed matter physics}, Mathematica Bohemica {\bf 140} (2015) 385--393.

\bibitem{n2} C. Escudero, F. Gazzola and I. Peral, {\it Global existence versus blow-up results
for a fourth order parabolic PDE involving  the Hessian}, J. Math. Pures Appl. {\bf 103} (2015) 924--957.

\bibitem{n3} C. Escudero, R. Hakl, I. Peral and P.~J. Torres, {\it On radial stationary solutions to a model of nonequilibrium growth},
Eur. J. Appl. Math. {\bf 24} (2013) 437--453.

\bibitem{n4} C. Escudero, R. Hakl, I. Peral and P.~J. Torres, {\it Existence and nonexistence results for a singular boundary value problem arising
in the theory of epitaxial growth}, Mathematical Methods in the Applied Sciences {\bf 37} (2014) 793--807.

\bibitem{escudero2} C. Escudero and E. Korutcheva, {\it Origins of scaling relations in nonequilibrium growth}, J. Phys. A: Math. Theor. {\bf 45} (2012) 125005.

\bibitem{n5} C. Escudero and I. Peral, \textit{Some fourth order nonlinear elliptic problems related to epitaxial growth},
J. Differential Equations {\bf 254} (2013) 2515--2531.

\bibitem{n6} C. Escudero and P. J. Torres, \textit{Existence of radial solutions to biharmonic k-Hessian equations},
J. Differential Equations {\bf 259} (2015) 2732--2761.

\bibitem{GGS} F. Gazzola, H. Grunau and G. Sweers, {\it Polyharmonic boundary value problems. Positivity preserving and nonlinear
higher order elliptic equations in bounded domains}, Lecture Notes in Mathematics, 1991. Springer-Verlag, Berlin, 2010.

\bibitem{KB} I. Kovacic and M. J. Brennan, {\it The Duffing Equation:
Nonlinear Oscillators and Their Behaviour}, John Wiley and Sons, New York,
2011.

\bibitem{LH} D. Li and H. Huang, \textit{Blow-up phenomena of second-order nonlinear differential equations},
J. Math. Anal. Appl. {\bf 276} (2002) 184--195.

\bibitem{opial} Z. Opial, \textit{Sur les int\'{e}grales born\'{e}es de l'\'{e}quation
$u''=f(t,u,u')$}, Annales Polonici Math. {\bf 4}, 314--324 (1958).

\end{thebibliography}
\end{document}